\documentclass{amsproc}
\usepackage{amsfonts}
\usepackage{amssymb}
\usepackage{amsthm}
\usepackage{amsmath}
\usepackage{graphicx}

\newtheorem{theorem}{Theorem}[section]
\newtheorem{lemma}[theorem]{Lemma}
\newtheorem{proposition}[theorem]{Proposition}
\newtheorem{fact}[theorem]{Fact}
\newtheorem{corollary}[theorem]{Corollary}
\newtheorem{conjecture}[theorem]{Conjecture}

\theoremstyle{definition}

\theoremstyle{remark}
\newtheorem{remark}[theorem]{Remark}

\numberwithin{equation}{section}

\begin{document}

\title{Linear operators with wild dynamics}

\author{Jean-Matthieu Aug\'{e}}
\address{Universit\'{e} Bordeaux 1
351, cours de la Lib\'{e}ration - F 33405 TALENCE cedex }
\email{jean-matthieu.auge@math.u-bordeaux1.fr}


\begin{abstract}
If $X$ is a separable infinite dimensional Banach space, we construct a bounded and linear operator $R$ on $X$ such that 
$$ A_R=\{x \in X, \|R^tx\| \rightarrow \infty\} $$ is not dense and has non empty interior with the additional property that $R$ can be written $I+K$, where $I$ is the identity and $K$ is a compact operator. This answers two recent questions of H\'ajek and Smith.
\end{abstract}

\maketitle

\section{Introduction}

Let $X$ be a Banach space (we will sometimes distinguish the real and complex cases) and $R$ a bounded linear operator on $X$. For $x$ in $X$, let 
$$O_R(x)=\{R^tx, t \geqslant 0\}$$ be the orbit of $x$ under the action of $R$. If there exists $x \in X$ such that $O_R(x)$ is dense in $X$, $R$ is called a hypercyclic operator (and $x$ a hypercyclic vector for $R$). This class of operator has been intensively studied in the last two decades, see \cite{BM} for a nice survey on this topic. Furthermore, the orbits are connected with the invariant subset problem which asks if there exists an operator on $X$ with non trivial invariant closed subset. Indeed, $R$ does not have any trivial invariant closed subset if and only if for each $x \neq 0$, $O_R(x)$ is dense in $X$. In the space $\ell^1$, Read \cite{R} constructed such an operator. But in the Hilbert space, the problem is still open. 
Here, we will be interested in the orbits which tend to infinity. The following is our main result.

\begin{theorem}\label{thm:main} Let $X$ be an infinite-dimensional separable, real or complex, Banach space.
There exists a linear continuous operator $R$ on $X$ such that, if we put 
$$A_R=\{x \in X, \|R^tx||\rightarrow \infty\} \; \; {\rm and} \; \; B_R=\{x \in X, \underline{\lim} \|R^tx-x\|=0\},$$
A) $A_R$ and $B_R$ have non-empty interior and $\{A_R, B_R\}$ form a partition of $X$. \\
B) $R$ can be written $I+K$, with $K$ a compact operator.
\end{theorem}

As a direct consequence, we obtain the following more general result.

\begin{corollary}\label{cor:main} The conclusion of Theorem \ref{thm:main}  remains true if one only assumes that $X$ has a closed and separable subspace of infinite dimension which is complemented in $X$.
\end{corollary}

\begin{proof} We write $X=X_0 \oplus Y$ where $X_0 \subset X$ is a closed separable Banach space of infinite dimension and $Y$ is a closed subspace of $X$. We then consider $R=R_0 \oplus I$ where $R_0$ is an operator on $X_0$ which satisfies Theorem \ref{thm:main}. Since $R^t=R_0^t \oplus I$, we get $A_{R}=A_{R_0}+Y$, $B_{R}=B_{R_0}+Y$ and it is clear that $R$ has the required properties.
\end{proof}

Theorem \ref{thm:main}  is motivated by a conjecture of Pr{\v{a}}jitur{\v{a}} \cite{P} which was recently solved negatively by H\'ajek and Smith \cite{HS}:

\begin{conjecture}Let $R$ be an operator on a Banach space, then $A_R$ is either empty or dense. 
\end{conjecture}

H\'ajek and Smith showed part A) of Theorem \ref{thm:main}  when $X$ has a symmetric basis. Therefore, they asked if the theorem could be true for a space which has only an unconditional basis (which is a weaker property than having a symmetric basis). We show that in fact no hypothesis of basis existence is necessary. Furthermore, a second question they asked was: is it possible to choose the operator $R$ of the form $I+K$, with $K$ a compact operator? The reason for this question is that they observed that for a compact operator $K$ and more generally, for a strictly singular operator $S$ (that is an operator which is not an isomorphism when restricted to any infinite-dimensional subspace), $A_S$ was either empty or dense. On the other hand, there exist some spaces with very few operators: Gowers and Maurey \cite{GM} constructed a Banach space where each operator has the form $\lambda I +S$, with $S$ strictly singular and even more recently, Argyros and Haydon \cite{AH} found a space where each operator has the form $\lambda I+K$, with $K$ compact. Thus, if some  operators satisfying part A) of Theorem \ref{thm:main} exist on any separable Banach space, it must be true that they can take the form $\lambda I+K$, and we show that it is indeed possible. \\

Note that if $(\|R^t\|)_{t \geqslant 0}$ is unbounded, by the uniform boundedness principle, 
$$ A'_R=\{x \in X, (\|R^tx\|)_{t \geqslant 0}\;\; {\rm is \;\; unbounded}\} $$ is a dense $G_{\delta}$ in $X$. So for any operator $R$, $A'_R$ is either dense or empty. However, this does not show anything for $A_R$ since there exist some operators such that $\|R^t\| \rightarrow \infty$ but for all $x$, $\|R^tx\| \nrightarrow \infty$. For example it is the case for the following weighted backward shift \cite{HS} defined on $\ell^p$ ($1 \leqslant p <\infty$):
$$ Be_i= \begin{cases}
\left(\dfrac{i}{i-1}\right)^{\frac{1}{p}}e_{i-1} \; \; {\rm if} \; \; i>1. \\
0 \; \; {\rm if} \; \; i=1.
\end{cases} \\ $$
However, M\"uller and Vr{\v{s}}ovsk\'y \cite{MV} showed that if $\|R^t\|$ grows fast enough, namely if $\sum_{k=1}^{\infty} \frac{1}{\|R^k\|}<\infty$ then $A_R$ is dense. \\

Let us explain the organization of this paper. Section 2 is devoted to a preliminary result: we prove a separating lemma in finite dimension (which might be of independant interest). We then use it in section 3 to construct an operator satisfying Theorem \ref{thm:main}  (we treat separately the complex and real cases). In section 4, we present a kind of version of the operator constructed by H\'ajek and Smith with a slightly simplified proof. In section 5, we make some comments on the two operators and we prove that in fact, there are many operators which fail Pr{\v{a}}jitur{\v{a}}'s conjecture.  \\ 

Let us settle some notation: if $X$ is a Banach space, we denote by $\mathcal{L}(X)$ the set of linear and bounded operators on $X$, by $\mathcal{G}\mathcal{L}(X) \subset \mathcal{L}(X)$ the set of invertible operators and by $\mathcal{A}_X$ the set of operators which satisfy part A) of Theorem \ref{thm:main}. Until the end of the paper $X$ will always be assumed to be infinite dimensional (if $X$ is finite dimensional, Pr{\v{a}}jitur{\v{a}}'s conjecture is true, see \cite{HS}).

\section{An asymptotic separating lemma in finite dimension}
Here $d \geqslant 2$ is an integer and $\mathbb{K}=\mathbb{R}$ or $\mathbb{C}$.

\begin{proposition}\label{prop:sep} Let $H \neq \emptyset$ be a closed subset of $\mathbb{K}^d$ such that $H$ is a union of linear hyperplanes. Then, there exists a sequence of linear forms $(f_n)$ on $\mathbb{K}^d$ such that \\
1) For all $x \notin H$, $\lim |f_n(x)|=\infty$. \\
2) For all $x \in H$, $\underline{\lim} |f_n(x)|=0.$
\end{proposition} 

We will do the proof in the real case. Let us denote by $S=S_{\mathbb{R}^d}$ the unit sphere of $\mathbb{R}^d$ and put $C=S \cap H$. We first prove the following lemma. 

\begin{lemma}\label{lem:net} There exists a constant $K>0$ and a sequence $(u_n) \subset C$ such that for each $x \in C$, there exists a sequence of integers $(p_n)$ which goes to infinity, and satisfies for each $n$, $p_n \leqslant n$ and $\|x-u_{p_n}\| \leqslant \frac{K}{n^{1/d-1}}$.
\end{lemma}

\begin{proof}[Proof of Lemma \ref{lem:net}] There exists $L>0$ such that for each $k$, we can find a $2^{-k}$ net $H_k$ of $C$ with $|H_k| \leqslant L(2^k)^{d-1}$ where $|H_k|$ denotes the number of elements of $H_k$. The sequence $(u_n)$ is obtained by enumerating first the elements of $H_1$, then all the elements of $H_2$ and so on. Let us fix $x \in C$ and check the property of the lemma. For $n \in \mathbb{N}$, there exists an integer $k$ such that $|H_1|+\ldots+|H_k| \leqslant n < |H_1|+\ldots+|H_k|+|H_{k+1}|$, in particular, we get
$$n \leqslant L \sum_{i=1}^{k+1} (2^{d-1})^{i}=2^{d-1}L \dfrac{(2^{d-1})^{k+1}-1}{2^{d-1}-1} \leqslant C_{d} 2^{k(d-1)}$$
where $C_d>0$ only depends on $d$. By the definition of $H_k$, there exists an element $u_{p_n} \in H_k$ such that $\|x-u_{p_n}\| \leqslant 2^{-k} \leqslant (\frac{C_d}{n})^{1/d-1}$. Since $k-1 \leqslant |H_1|+\ldots+|H_{k-1}| < p_n \leqslant |H_1|+\ldots+|H_k|$, $p_n \leqslant n$, $\lim p_n=\infty$ and the lemma is shown. 
\end{proof}

\begin{proof}[Proof of Proposition \ref{prop:sep}] For each $n$, we can select a linear form $f_n$ such that $f_n(u_n)=0$,   ${\rm Ker} f_n \subset H$ and $\|f_n\|=n^{1/2(d-1)}$. If $x \notin H$, then 
$$ |f_n(x)|=\|f_n\| d(x,{\rm Ker} f_n) \geqslant \|f_n\| d(x,H) \geqslant n^{1/2(d-1)} d(x,H)$$
thus we have 1). For 2), we may assume that $x \neq 0$ and apply the lemma with $y=\frac{x}{\|x\|} \in C$, we obtain a sequence $(p_n)$ with $p_n \leqslant n$,  $\lim p_n=\infty$ and $\|y-u_{p_n}\| \leqslant \frac{K}{n^{1/d-1}}$. For each $n$ we have 
$$ |f_{p_n}(x)|=\|x\| |f_{p_n}(y-u_{p_n})| \leqslant \|f_{p_n}\| \|y-u_{p_n}\| \|x\| \leqslant \dfrac{K}{n^{1/2(d-1)}} \|x\|$$ 
which proves 2) and achieves the proof. 
\end{proof}

For the complex case, the proof is almost the same. The only difference is that for $\epsilon>0$ the entropy of the unit sphere is of the order of $\frac{1}{\epsilon^{2d-1}}$ which gives a different estimate of $\|x-u_{p_n}\|$ in the lemma, so we have also to change the exponent in $\|f_n\|$ in the proof of the proposition.

\section{Construction of Operators}

\subsection{Complex case}

In this sub-section, we prove the complex version of Theorem \ref{thm:main}.

\begin{theorem}\label{thm:complex} Let $X$ be a separable complex Banach space.
Then, there exists $R \in \mathcal{A}_X$ such that $R-I$ is nuclear.
\end{theorem}

Recall that an operator $N$ is nuclear if there exist some continuous linear functionals $x^*_k$ and some vectors $x_k$ such that $N$ can be written as $N=\sum_{k=1}^{\infty} x_k \otimes x^*_k$ with $\sum_{k=1}^{\infty} \|x^*_k\|\|x_k\|<\infty$ where $x_k \otimes x^*_k$ is the rank one operator defined by
$$ x_k \otimes x^*_k(x)=\langle x^*_k,x \rangle x_k.$$
 These operators are obviously compact, being limits of finite rank operators. \\

To construct our operator, we will use the following theorem which says that a separable Banach space supports a sequence $(e_n)_{n \geqslant 1}$ which behaves more or less like a Schauder basis \cite{OP}.

\begin{theorem} \label{thm:op} If $X$ is a separable Banach space and $\epsilon>0$, one can find sequences $(e_n,e_n^*) \subset X \times X^*$ such that: \\
i) ${\rm span}(e_n,n \geqslant 1)$ is dense in $X$. \\
ii) $\langle e_n^*,e_m \rangle=\delta_{n,m}$ where $\delta_{n,m}=0$ if $n \neq m$ and $1$ if $n=m$.  \\
iii) $\sup \|e_n\| \|e_n^*\| \leqslant 1+\epsilon$.
\end{theorem}

For our case, we won't need the precise estimate with $\epsilon$. However, observe that by replacing $e_n$ by $e_n/\|e_n\|$ and $e_n^*$ by $\|e_n\|e_n^*$ we can suppose that $\|e_n\|=1$, $K=\sup\|e_n^*\|<\infty$ and properties i) and ii) of Theorem \ref{thm:op} are still satisfied. We then consider a closed subset $F \subset {\rm span}(e_1,e_2)$ which is a union of complex lines passing through 0 and the sequence $(f_k)$ of linear forms associated on ${\rm span}(e_1,e_2)$  in the asymptotic separating lemma (we just use the particular case of the lemma $d=2$). Let us denote by $P$ the projection defined on $X$ onto ${\rm span}(e_1,e_2)$ by 
$$ Px=\langle e_1^*,x \rangle e_1 +\langle e_2^*,x \rangle e_2. $$
$P$ is continuous with $\|P\| \leqslant 2K$. Then, we put $\lambda_1=\lambda_2=1$ and for $k \geqslant 3$, $\lambda_k=e^{\frac{i \pi}{m_k}}$ where $(m_k)$ is a sequence of positive integers such that $m_{k} | m_{k+1}$ and which goes very fast to infinity, namely 
$$ \sum_{k=3}^{\infty} \dfrac{m_{k-2}}{m_{k-1}} \|f_k\|<\infty \; \; {\rm and} \; \; m_k \geqslant 15m_{k-1}. $$
We denote by $c_{00}$ the subspace generated by the sequence $(e_n)_{n \geqslant 1}$. For $x \in c_{00}$, $x=\sum_{k=1}^{n} x_k e_k$, we set
$$ Sx=\sum_{k=1}^{n} \lambda_k x_k e_k.$$
This defines a bounded operator $S$ on $c_{00}$, indeed for $x=\sum_{k=1}^n x_ke_k \in c_{00}$,
$$ \|Sx\| \leqslant \sum_{k=1}^n |\lambda_k-1| \|x_k e_k\|+\|x\| \leqslant \left(K \sum_{k=1}^{\infty} |\lambda_k-1|+1\right) \|x\| $$
and by our assumption on $(m_k)$, it is clear that $\sum_{k=1}^{\infty} |\lambda_k-1|<\infty$. Hence, $S$ extends to a bounded operator on $X$ (because $c_{00}$ is dense in $X$). We still denote this operator by $S$.

\begin{remark} Note that if $(e_k)$ is a normalized Schauder basis, then if $x=\sum_{k=1}^{\infty} x_k e_k$, $S$ is explicitly given by
$$ Sx=\sum_{k=1}^{\infty} \lambda_k x_k e_k.$$
\end{remark}
 
For $x \in X$, we now define an operator $R$ by 
$$ Rx=Sx+\sum_{k=3}^{\infty} \dfrac{1}{m_{k-1}} f_k(Px) e_k. $$
Note that the second sum is a bounded operator because $\sum_{k=3}^{\infty} \frac{1}{m_{k-1}} \|f_k\| < \infty$. So $R$ is itself bounded, if we set 
$$\tilde{A}_R=\{x \in X, Px \notin F\} \; \; {\rm and} \; \; \tilde{B}_R=\{x \in X, Px \in F\},$$
then we will show that $A_R=\tilde{A}_R$ and $B_R=\tilde{B}_R$ so $R$ will be the required operator (provided the relative interior of $F$ in ${\rm span}(e_1,e_2)$ is non-empty and $F \neq {\rm span}(e_1,e_2)$).
Since $\lambda_k^{2m_k}=1$ and more generally $\lambda_k^{2m_n}=1$ for $n \geqslant k$ (by the hypothesis of divisibility of the $m_k$), the operator $S$ has the following property.

\begin{lemma} For $x \in X$, $ \lim_{n \rightarrow \infty} S^{2m_n}x=x.$
\end{lemma}

\begin{proof} We first show that the sequence $(\|S^{2m_l}\|)$ is bounded. If we fix $x=\sum_{k=1}^{n} x_k e_k$ such that $\|x\|=1$, we have
\begin{eqnarray*}  \left \|S^{2m_l}x-x \right\| & = & \left \| \sum_{k=1}^{n} (\lambda_k^{2m_l}-1) x_k e_k \right \| \\
& = & \left\| \sum_{k=l+1}^{n} (\lambda_k^{2m_l}-1) x_k e_k \right\| \; \; {\rm because} \; \; \lambda_k^{2m_l}=1 \; \; {\rm for} \; \; l \geqslant k \\
& \leqslant & K \sum_{k=l+1}^{n} \left|\lambda_k^{2m_l}-1 \right| \\
& \leqslant & 2K \sum_{k=l+1}^{\infty} \dfrac{2\pi m_l}{m_k} \;\; {\rm because} \;\; |e^{z}-1| \leqslant 2|z| \;\; {\rm for} \;\; |z| \leqslant 1/2 \\
& \leqslant &  K \sum_{k=1}^{\infty}  \dfrac{4 \pi m_{k-1}}{m_k}.
\end{eqnarray*} 

So $(\|S^{2m_l}x\|)$ is bounded by $1+K \sum_{k=1}^{\infty}  \frac{4 \pi m_{k-1}}{m_k} < \infty$ whenever $x \in c_{00}$, $\|x\|=1$. By density of $c_{00}$ in $X$, this implies that the sequence $(\|S^{2m_l}\|)$ is bounded. To see the claim of the lemma, fix $x \in X$, $\epsilon>0$ and $y \in c_{00}$ such that $\|y-x\| \leqslant \epsilon$, we have
\begin{eqnarray*} \left \|S^{2m_l}x-x \right\| & = & \left \|S^{2m_l}(x-y)+(S^{2m_l}y-y)+(y-x) \right \| \\
& \leqslant & \sup \|S^{2m_l}\|\epsilon+\|S^{2m_l}y-y\|+\epsilon.
\end{eqnarray*}
Since for large enough $l$, $S^{2m_l}y=y$, we have the conclusion.
\end{proof}

We now explicit the iterates of $R$. 

\begin{lemma} For $x \in X$ and $t \geqslant 1$, we have
$$ R^t{x}=S^{t}x+\sum_{k=3}^{\infty} \dfrac{\lambda_{k,t}}{m_{k-1}}f_k(Px) e_k \; \; {\rm where} \; \; \lambda_{k,t}=\sum_{l=0}^{t-1} \lambda_k^l. $$
\end{lemma}

\begin{proof} The statement is true for $t=1$, we proceed by induction and suppose it is true for $t$, we have
\begin{eqnarray*} R^{t+1}x=R(R^{t}x) & = & R(S^{t}x)+R\left(\sum_{k=3}^{\infty} \dfrac{\lambda_{k,t}}{m_{k-1}}f_k(Px) e_k\right) \\
& = & S^{t+1}x+\sum_{k=3}^{\infty} \dfrac{1}{m_{k-1}}f_k(PS^{t}x)e_k+\sum_{k=3}^{\infty} \dfrac{\lambda_k \lambda_{k,t}}{m_{k-1}}f_k(Px) e_k \\
& + & \sum_{k=3}^{\infty}\dfrac{1}{m_{k-1}}f_k(Py) e_k
\end{eqnarray*}
where $y=\sum_{l=3}^{\infty} \frac{\lambda_{l,t}}{m_{l-1}}f_l(Px)e_l$, so $Py=0$. Since $\lambda_1=\lambda_2=1$, we have $PS^{t}x=Px$, and since $\lambda_k \lambda_{k,t}+1=\lambda_{k,t+1}$, we deduce the formula for $t+1$. 
\end{proof}

We now study the behaviour of $\lambda_{k,t}$. We have the following easy fact.

\begin{fact} i) For all $k,t \geqslant 1$, $|\lambda_{k,t}| \leqslant t$. \\
ii) $\lambda_{k,2m_n}=0$ for $n \geqslant k$. \\
iii) For $m_{k-1} \leqslant t \leqslant m_{k}$, $|\lambda_{k,t}| \geqslant \frac{2}{\pi} m_{k-1}$.
\end{fact}

\begin{proof} i) is just the triangle inequality. ii) is obvious since $\lambda_{k,t}$ can be computed easily as a geometric sum:
$$\lambda_{k,t}=\dfrac{e^\frac{i\pi t}{m_k}-1}{e^\frac{i\pi}{m_k}-1}.$$
For iii), we get 
$$ | \lambda_{k,t}|=\dfrac{\left|\sin\left(\frac{\pi t}{2 m_k}\right)\right|}{\left|\sin\left(\frac{\pi}{2m_k}\right)\right|}.$$
The desired lower bound follows from the inequalities : $\sin(y) \geqslant \frac{2}{\pi}y$ for $y \in [0,\frac{\pi}{2}]$ and $|\sin(y)| \leqslant |y|$ for $y \in \mathbb{R}$.
\end{proof}

We can now pass to the proof of the main theorem. Let $x \in c_{00} \cap \tilde{A}_R, x=\sum_{k=1}^{n} x_ke_k$. Fix $k \geqslant 3$, we set $\delta_{k \leqslant n}=1$ if $k \leqslant n$ and $\delta_{k \leqslant n}=0$ if $k>n$. We have
\begin{eqnarray*}  K \|R^tx\|  \geqslant  |\langle e_k^*,R^tx \rangle| &=& \left|\lambda_k^t\langle e_k^*,x \rangle \delta_{k \leqslant n} +\dfrac{\lambda_{k,t}}{m_{k-1}}f_k(Px)\right| \\
& \geqslant & \left|\dfrac{\lambda_{k,t}}{m_{k-1}}f_k(Px)\right|-| \lambda_k^t \langle e_k^*,x \rangle \delta_{k \leqslant n} | \\
& \geqslant & \left|\dfrac{\lambda_{k,t}}{m_{k-1}}f_k(Px)\right|-K \|x\|.
\end{eqnarray*}

We claim that $\tilde{A}_R \cap c_{00}$ is dense in $\tilde{A}_R$. Since the last inequality involves only continuous functions, it stays true for $x \in \tilde{A}_R$ and for $k \geqslant 3$, $t \geqslant 1$. Now, let us fix $k\geqslant 1$ and $m_{k-1} \leqslant t \leqslant m_k$, using the last inequality and the previous fact, we get
$$ \|R^tx\| \geqslant \left| \dfrac{\lambda_{k,t}}{Km_{k-1}}f_k(Px) \right|-\|x\| \geqslant \dfrac{2}{\pi K} |f_{k}(Px)|-\|x\|.$$
This proves that $\|R^tx\| \rightarrow \infty$ whenever $x \in \tilde{A}_R$ because by the separating lemma $|f_k(Px)| \rightarrow \infty$. For the proof of the claim, if $x \in \tilde{A}_R$, there exists $(z_k) \subset c_{00}$ such that $x=\lim z_k$. Write $z_k=z_{1,k}e_1+z_{2,k}e_2+y_k$ with $y_k \in {\rm span}(e_l, l \geqslant 3)$. If $Px=x_1e_1+x_2e_2$, the sequence $\tilde{z_k}=x_1e_1+x_2e_2+y_k$ has the required properties. \\
 
Now, let $x \in \tilde{B}_R$, then there exists a subsequence $(k_n)$ such that $|f_{k_n}(Px)|$ goes to 0. Write

\begin{eqnarray*} \left\|\sum_{k=3}^{\infty} \dfrac{\lambda_{k,2m_{k_n-1}}}{m_{k-1}}f_k(Px) e_k \right\|   \leqslant &  & \sum_{k=3}^{k_n-1} \dfrac{\left| \lambda_{k,2m_{k_n-1}} \right|}{m_{k-1}}|f_k(Px)|
+\dfrac{\left|\lambda_{k_n,2m_{k_n-1}}\right|}{m_{k_n-1}}|f_{k_n}(Px)| \\
& + & \sum_{k=k_n+1}^{\infty} \dfrac{|\lambda_{k,2m_{k_n-1}}|}{m_{k-1}}  |f_k(Px)|. 
\end{eqnarray*}

The first sum is equal to 0 according to the previous fact. Using the fact again, we see that the second term goes to $0$ because $\frac{|\lambda_{k_n,2m_{k_n-1}}|}{m_{k_n-1}} \leqslant 2$. For the last sum, $k \geqslant k_{n}+1$ so $k_{n}-1 \leqslant k-2$, and $m_{k_n-1} \leqslant m_{k-2}$, hence
\begin{eqnarray*}  \sum_{k=k_n+1}^{\infty} \dfrac{\left|\lambda_{k,2m_{k_n-1}}\right|}{m_{k-1}}|f_k(Px)| & \leqslant & \left(\sum_{k=k_n+1}^{\infty} \dfrac{2m_{k_n-1}}{m_{k-1}}\|f_k\|\right) \|P\| \|x\| \\
& \leqslant & \left( \sum_{k=k_n+1}^{\infty} \dfrac{2m_{k-2}}{m_{k-1}}\|f_k\|\right) \|P\| \|x\| .
\end{eqnarray*}

Since $\sum_{k \geqslant 1}\frac{m_{k-2}}{m_{k-1}} \|f_k\|<\infty$, the last term goes to 0. On the other hand, $S^{2m_{k_n-1}}x$ goes to $x$, this shows that $\underline{\lim} \|R^tx-x\|=0$ whenever $x \in \tilde{B}_R$. \\ 

Finally, let us check the last claim of our theorem, by a density argument we see that for $x \in X$
$$(S-I)x=\sum_{k=3}^{\infty} \left(\lambda_k-1\right) \langle e_k^*,x \rangle e_k.$$
Hence
$$N=R-I=\sum_{k=3}^{\infty} e_k \otimes \left((\lambda_k-1)e_k^*+\dfrac{1}{m_{k-1}}f_k P\right).$$
Since $\sum_{k=3}^{\infty} |\lambda_k-1|<\infty$, $\sum_{k=3}^{\infty} \frac{1}{m_{k-1}}\|f_k\|<\infty$, and $(e_k)$ and $(e_k^*)$ are bounded, $N$ is nuclear and $R=I+N$.

\begin{remark} The statement of our theorem can be sharpened:  it is apparent from the definition of $R$ that for any $\epsilon>0$, one can choose the sequence $(m_k)$ to have $\|R-I\| \leqslant \epsilon$. This proves that $I$ is in the norm closure of $\mathcal{A}_X$.
\end{remark}

\begin{remark} Although Corollary \ref{cor:main} does not apply to $\ell^{\infty}$, we can see that our example works in $\ell^{\infty}$. We thus conjecture that operators with the properties stated in Theorem \ref{thm:main} can be constructed in arbitrary Banach spaces.
\end{remark}

\subsection{Real case}
We now show the version of Theorem \ref{thm:main} for real spaces. 

\begin{theorem} Let $X$ be a separable real Banach space.
Then, there exists $R \in \mathcal{A}_X$ such that $R-I$ is nuclear.
\end{theorem}

We keep the same notations than in the complex case for $F$, $(f_k)$, $c_{00}$, $(e_k)$,$(e_k^*)$, $K$ and $P$. We define $S$ on $c_{00}$ by the following formulas $Se_1=e_1$, $Se_2=e_2$ and for $k \geqslant 2$:
$$S{e_{2k-1}}=\cos(\theta_k)e_{2k-1}+\sin(\theta_k)e_{2k} \; \; {\rm and } \; \; S{e_{2k}}=-\sin(\theta_k)e_{2k-1}+\cos(\theta_k)e_{2k} $$
where $\theta_k=\frac{\pi}{m_k}$ and $(m_k)$ satisfy the same hypotheses of growth and divisibility as in the complex case. Thus in the (algebraic) basis $(e_k)$, $S$ is a block diagonal operator built from rotation matrices. We check that $S$ is bounded on $c_{00}$, so it extends to $X$ and we check that $S^{2m_n}x \rightarrow x$. We then define $R$ (the required operator) by
$$ Rx=Sx+\sum_{k=2}^{\infty} \dfrac{1}{m_{k-1}} f_k(Px) e_{2k}.$$
The computation of the iterates of $R$ gives 
$$ R^{t}x=S^{t}x+\sum_{k=2}^{\infty} \dfrac{1}{m_{k-1}} f_k(Px)\mu_{k,t}$$
where $\mu_{k,t}=\left(\sum_{l=0}^{t-1} \cos (l \theta_k)\right) e_{2k}-\left(\sum_{l=0}^{t-1} \sin (l \theta_k)\right) e_{2k-1}$. We thus have $\mu_{k,t}=0 \Leftrightarrow \lambda_{k,t}=0$, $\|\mu_{k,t}\| \leqslant 2t$ and considering the functionals $e_{2k-1}^*$ and $e_{2k}^*$:
$$ \| \mu_{k,t} \| \geqslant \dfrac{1}{K} \max \left(\left|\sum_{l=0}^{t-1} \cos (l \theta_k)\right|,\left|\sum_{l=0}^{t-1} \sin (l \theta_k)\right|\right) \geqslant \dfrac{1}{\sqrt{2}K} |\lambda_{k,t}|.$$
So the inequalities for $\mu_{k,t}$ are the same than $\lambda_{k,t}$ (up to a multiplicative constant) which allows to follow the previous proof (in the complex case) without other changes.

\section{The construction of H\'ajek and Smith}
We present a slightly modified version of the example of H\'ajek and Smith. The following theorem is of course a particular case of our above theorem. However, we will outline some differences between the two operators in the last section.

\begin{theorem} Let $X$ be a Banach space (real or complex) with a symmetric basis. Then, there exists $R \in \mathcal{A}_X$.
\end{theorem} 

Recall that a basis $(e_k)$ is said to be symmetric if for every permutation $\pi$ of the integers, $(e_{\pi(k)})$ is a basis and is equivalent to $(e_k)$. If for $x=\sum_{k=1}^{\infty} x_ke_k \in X$, we set 
$$ \|x\|_0=\sup_{|\lambda_k| \leqslant 1} \sup_{\pi} \left \| \sum_{k=1}^{\infty} x_k \lambda_k e_{\pi(k)} \right \|,$$
then $\|.\|_{0}$ is an equivalent norm on $X$ which satisfies for every $\pi$, 
$$\left \| \sum_{k=1}^{\infty} x_ke_{\pi(k)} \right \|_{0}=\left \| \sum_{k=1}^{\infty}|x_k| e_{k} \right \|_{0}$$
 and whenever $0 \leqslant x_k \leqslant y_k$, 
 $$\left \| \sum_{k=1}^{\infty} x_k e_k \right \|_{0} \leqslant \left \| \sum_{k=1}^{\infty} y_k e_k \right \|_{0}.$$ In the rest of this section, we suppose that our norm $\|.\|$ has the two precedent properties (such a norm is sometimes called a symmetric norm). \\
  
In the complex case, our operator $S$ was built with roots of unity, $S$ will now be a cycle operator. As in the original paper of H\'ajek and Smith, we will prove some finite-dimensional estimates. However, H\'ajek and Smith make some estimates in $\ell^p$ and $c_0$ and then combine it with a theorem of Tzafriri \cite{T}. We will give a direct and short proof of these estimates.

\subsection{Local estimates}
Let $Z$ be a space of dimension $H \geqslant 1$ with a symmetric basis ($e_1,\ldots,e_H)$ such that $H \geqslant 4m$ ($m$ and $H$ are positive integers). We define an operator $S$ on $Z$ by $Se_i=e_{i+1}$ for $1 \leqslant i \leqslant H-1$ and $Se_H=e_1$. We will be interested in the behaviour of the sequence: $w_t=(I+\ldots+S^{t-1})w$ where $w=\sum_{i=1}^{m}e_i-\sum_{i=m+1}^{2m}e_i$ (this is motivated by the fact that ($w_t$) will appear in the study of the iterates of the operator in the general case). We have:

\begin{lemma} i) $w_{H}=0$ and the sequence $(w_t)_{t \geqslant 1}$ is $H$-periodic. \\
ii) If $2m \leqslant t \leqslant H-2m$, $\|w_{t}\|$ attains its maximum and 
$$ \|w_{t}\| \geqslant \dfrac{m}{2} \|w\|. $$

\end{lemma}

\begin{proof} i) $S$ is a cycle operator of order $H$, thus for $1 \leqslant i \leqslant H$, we have $(I+S+\ldots+S^{H-1})e_{i}=\sum_{i=1}^{H} e_{i}$, and if we put $s=\sum_{i=1}^{H} e_{i}$, we get $w_{H}=ms-ms=0$. \\
Moreover, since $S^{H}$ is the identity, we have 
$$w_{t+H}=(I+S+\ldots+S^{H-1})w+(I+\ldots+S^{t-1})w=(I+\ldots+S^{t-1})w=w_{t}.$$

ii) An easy computation shows that for $2m \leqslant t \leqslant H-2m$, $w_t$ has two bumps of height $m$ and of length $2m-1$, one with positive terms, the other with negative terms:
\begin{eqnarray*} w_t= & e_1 & +2e_2+\ldots +me_m+(m-1)e_{m+1}+\ldots e_{2m-1} \\ 
& - & (e_{t+1}+2e_{t+2}+\ldots +me_{t+m}+\ldots e_{t+2m-1}).
\end{eqnarray*}

Because the norm is symmetric, we can arrange the terms in a decreasing way and cancel the signs, hence 
\begin{eqnarray*} \|w_t\|=&\|&\! \! \! \! \! \! m(e_1+e_2)+(m-1)(e_3+e_4+e_5+e_6)+\ldots \\ & + & \! \! \! \! \! \!(e_{4m-5}+e_{4m-4}+e_{4m-3}+e_{4m-2})\|.  
\end{eqnarray*}
If we keep the coordinates which are superior to $\frac{m}{2}$, we get 
$$ \|w_t\| \geqslant \dfrac{m}{2} \|e_1+\ldots+ e_{2m}\|=\frac{m}{2}\|w\|.$$
We have also shown that if $u_t$ is the decreasing rearrangement of $|w_t|$, then for $2m \leqslant t \leqslant H-2m$ and $2m \leqslant s \leqslant H-2m$, $u_t=u_s$. Another computation shows that for $1 \leqslant s \leqslant 2m$ and $2m \leqslant t \leqslant H-2m$, $u_s \leqslant u_t$ and hence that $\|w_s\|=\|u_s\| \leqslant \|u_t\|=\|w_t\|$ because the norm is symmetric. For $H-2m \leqslant s \leqslant H$, we obtain the same estimates as in the case $1 \leqslant s \leqslant 2m$. This shows that $\|w_t\|$ attains its maximum for $2m \leqslant t \leqslant H-2m$. 
\end{proof}

\subsection{The operator of H\'ajek and Smith}
We assume that $X$ is a Banach space with a symmetric and normalized basis $(e_n)$ and write
$$ X={\rm span}(e_1,e_2) \oplus \overline{{\rm span}}(e_n, n \geqslant 3)= {\rm span}(e_1,e_2) \oplus Y.$$
Next, we are going to define $S$ on $X$ (we give it the same name as in the finite dimensional case because it will be a sum of these operators on finite blocks). We choose some sequences $(H_k)$ and $(m_k)$ of positive increasing integers such that $H_k \geqslant 4m_k$ and $H_k | H_{k+1}$. We define $\sigma_k$ the following permutation of order $H_k$: $\sigma_k=(1,2,\ldots,H_k)$, that is $\sigma_k(i)=i+1$ for $1 \leqslant i \leqslant  H_k-1$ and $\sigma_k(H_k)=1$. We enumerate the basis vectors of $Y$ as follows
$$ {\rm span}(e_n, n \geqslant 3)={\rm span}(e_{i,1}, 1 \leqslant i \leqslant H_1) \oplus {\rm span}(e_{i,2}, 1 \leqslant i \leqslant H_2) \oplus \ldots $$
So, each element $x$ of $X$ can be written 
$$ x=x_1 e_1 + x_2 e_2 + \sum_{k=1}^{\infty} \sum_{i=1}^{H_k} y_{i,k} e_{i,k} $$
and we put 
$$ Sx=x_1 e_1 + x_2 e_2 + \sum_{k=1}^{\infty} \sum_{i=1}^{H_k} y_{i,k} e_{\sigma_k(i),k}. $$
$S$ has the following properties:

\begin{lemma} i) $S$ is well defined and is an isometry. \\
ii) For $x \in X$, $\lim \|S^{H_k}x-x\|=0$.
\end{lemma}

\begin{proof}
i) is because the norm is symmetric. For ii), it is enough to check it for $y \in Y$. Fix $\epsilon>0$, there exists an integer $N$ such that 
$$ \left \| \sum_{l=N+1}^{\infty} \sum_{i=1}^{H_l} y_{i,l} e_{i,l} \right\| \leqslant \epsilon, \; \; {\rm where} \; \; y=\sum_{l=1}^{\infty} \sum_{i=1}^{H_l} y_{i,l} e_{i,l}. $$
For $k \geqslant N$, $H_k$ is a multiple of $H_N$, so for $1 \leqslant l \leqslant N$ and $1 \leqslant i \leqslant H_l$, $S^{H_k}(e_{i,l})=e_{i,l}$, hence  
$$ S^{H_k}y=\sum_{l=1}^{N} \sum_{i=1}^{H_l} y_{i,l} e_{i,l} + S^{H_k}\left(\sum_{l=N+1}^{\infty} \sum_{i=1}^{H_l} y_{i,l} e_{i,l}\right)$$
then 
$$ S^{H_k}y-y=-\sum_{l=N+1}^{\infty} \sum_{i=1}^{H_l} y_{i,l} e_{i,l}+S^{H_k}\left(\sum_{l=N+1}^{\infty} \sum_{i=1}^{H_l} y_{i,l} e_{i,l}\right).$$
Since $S$ is an isometry, we get for $k \geqslant N$: $\|S^{H_k}y-y\| \leqslant 2\epsilon$.
\end{proof}

We now define the operator $R=R_{HS}$ (we use the same notation and assumptions as in the previous section for $F$, $(f_k)$ and $P$, and the definition of the sets $\tilde{A}_R$ and $\tilde{B}_R$ also remain unchanged). For $x \in X$, we put    

$$ Rx=Sx+\sum_{k=1}^{\infty} f_k(Px) v_k $$ where $$v_k=\epsilon_k\left(\sum_{i=1}^{m_k} e_{i,k} - \sum_{i=m_k+1}^{2 m_k} e_{i,k}\right).$$ 
The iterates of $R$ are given by (same proof as in previous section)

\begin{lemma} $R^{t}x=S^{t}x+\sum_{k=1}^{\infty} f_k(Px) v_{k,t}$ whenever $x \in X$, $t \geqslant 1$, with $v_{k,t}=(I+S+\ldots S^{t-1})v_k$.
\end{lemma}

Now, let us choose the various constants: first, we introduce a sequence of positive and increasing integers ($a_k$) such that 
$$ \sum_{k=1}^{\infty} \dfrac{\|f_k\|}{a_k} < \infty \;\; {\rm and} \;\; 1+a_k \; | \; 1+a_{k+1} .$$
Then, we put  $m_k=a_1 \ldots a_k$, $H_k=2(m_k+m_{k+1})$. We have 
$$ \dfrac{H_{k+1}}{H_k}=\dfrac{m_{k+1}+m_{k+2}}{m_k+m_{k+1}}=\dfrac{\prod_{i=1}^{k+1}a_i (1+a_{k+2}) }{\prod_{i=1}^{k}a_i(1+a_{k+1})}=a_{k+1} \dfrac{1+a_{k+2}}{1+a_{k+1}}$$
so we get $H_k | H_{k+1}$. \\  

The constant $\epsilon_k$ is chosen such that $\|v_{k,2m_k}\|=1$. The finite dimensional estimates give $\|v_k\| \leqslant \frac{2}{m_k}$. We deduce from this that $R$ is well defined and bounded because
$$ \left \| \sum_{k=1}^{\infty} f_k(Px)v_k \right \| \leqslant \left( \sum_{k=1}^{\infty} 2 \dfrac{\|f_k\|}{m_k} \right) \|P\| \|x\| \leqslant \left( \sum_{k=1}^{\infty} 2 \dfrac{\|f_k\|}{a_k} \right) \|P\| \|x\|.  $$
The following lemma is now a reformulation of the finite dimensional estimates. 

\begin{lemma}\label{lem:reformul} i) For fixed $k$, $v_{k,H_k}=0$ and the sequence $(v_{k,t})_t$ is $H_k$-periodic. \\
ii) For $2m_k \leqslant t \leqslant H_k-2m_k$, $\|v_{k,t}\|$ is maximal (in $t$) and $\|v_{k,t}\|=1$. \\
iii) For all $t$, $\|v_{k,t}\| \leqslant \frac{2t}{m_k}$.
\end{lemma}

\begin{proof} i) and ii) were proved before. For iii), use the triangle inequality and $\|v_k\| \leqslant \frac{2}{m_k}$.
\end{proof}

The proof of the theorem is now very close to the one we produced before. First, because the norm is symmetric and the $(v_{k,t})_{k \geqslant 1}$ have disjoint supports, we get for $k \geqslant 1$ and $x \in X$
$$ \left\|\sum_{l=1}^{\infty} f_l(Px) v_{l,t} \right\| \geqslant  |f_k(Px)| \|v_{k,t}\|.$$ 
For each $t$, there exists $k$ such that $2m_k \leqslant t \leqslant 2m_{k+1}=H_k-2m_k$, thus Lemma \ref{lem:reformul}  gives  $\|v_{k,t}\|=1$ and we see that $\|R^tx\| \rightarrow \infty$ when $x \in \tilde{A}_R$. When $x \in \tilde{B}_R$, the proof is also in the same spirit as in section 3 (divide the sum into three parts) and if $|f_{k_n}(Px)| \rightarrow 0$, check that $R^{H_{k_n-1}}x \rightarrow x$ using Lemma \ref{lem:reformul}. Details are left to the reader.  

\section{Further results and remarks}

\subsection{Difference between the operators}

We denote by $\sigma(R)$ the spectrum of an operator $R$, by $\sigma_p(R)$ the set of its eigenvalues, by $r(R)$ its spectral radius, and we only consider complex spaces. Let also $\overline{\mathbb{D}}$ denote the unit closed complex disk and $\partial \mathbb{D}$ the unit complex circle. Let us recall that $\mathcal{A}_X$ is the set of operators which satisfy part A) of Theorem \ref{thm:main} and that by M\"uller and Vr{\v{s}}ovsk\'y's result (see the introduction), if $\sum_{k=1}^{\infty} \frac{1}{\|R^k\|}<\infty$ then $A_R$ is dense. Suppose that $R \in \mathcal{A}_X$, then we must have $r(R)=1$. Indeed, if $r(R)<1$, by the spectral radius formula, $\|R^k\| \rightarrow 0$ which of course contradicts $A_R \neq \emptyset$. If $r(R)>1$, then if we fix $1<a<r(R)$, we get for $k$ large enough $\|R^k\| \geqslant a^k$ (still with the spectral radius formula), thus $\sum_{k=1}^{\infty} \frac{1}{\|R^k\|}<\infty$ and $A_R$ is dense which is impossible. In particular, any operator $R \in \mathcal{A}_X$ has a spectrum contained in the unit closed disk and its spectrum has to intersect the unit circle. More precisely, in the two studied examples (which are built with a strong analogy), we have: 

\begin{proposition} Assume that in the construction of the operator $R$ from section 3, $(e_k)$ is a normalized Schauder basis, then $\sigma(R)=\{\lambda_k, k \geqslant 1\}$. On the other hand, if we take the operator $R_{HS}$ from section 4, then $\partial \mathbb{D} \subset \sigma(R_{HS}) \subset \overline{\mathbb{D}}$.
\end{proposition} 

\begin{proof} Our operator can be written as $I+N$, thus we have $\sigma(R)=\{1+\lambda, \lambda \in \sigma(N)\}$. Since $N$ is compact, $\sigma(N)=\sigma_p(N) \cup \{0\}$. Let $\lambda \in \sigma_p(N)$ and $x \neq 0, x=\sum_{k=1}^{\infty} x_k e_k$ such that $Nx=\lambda x$, then
$$ \begin{cases}0=\lambda x_k \; \; {\rm for} \; \; k \in \{1,2\}.
\\(\lambda_k-1)x_k+\dfrac{f_k(Px)}{m_{k-1}}=\lambda x_k \; \; {\rm for} \; \; k \geqslant 3.  \end{cases} \\ $$
If $\lambda \neq 0$, then $x_1=x_2=0$, that is $Px=0$ and we get that $\lambda=\lambda_k-1$ for some $k \geqslant 3$. On the other hand, $Ne_k=(\lambda_k-1) e_k$ for $k \geqslant 3$, hence $\sigma(N)=\sigma_p(N) \cup \{0\}=\{\lambda_k-1, k \geqslant 1\}$ and $\sigma(R)=\{\lambda_k, k \geqslant 1\}$. \\
Now, the operator $S_{HS}$ of section 4 (associated to $R_{HS}$) has an infinite matrix which is block diagonal with identity on first block, and then some matrices of Frobenius which have $X^{H_k}-1$ as characteristic polynomial on the $k$-th block (i.e ${\rm span}(e_{i,k}, 1 \leqslant i \leqslant H_k)$) , so by taking an eigenvector associated to a $H_k$ root of the unity for the matrix of Frobenius on the $k$-th block and completing it with some zeros, we see that 
$$ G=\bigcup_{k=1}^{\infty} \mathbb{U}_{H_k} \subset \ \sigma_p(R_{HS}) \subset \sigma(R_{HS}) $$
where $\mathbb{U}_{H_k}$ denotes the group of the $H_k$-th roots of the unity. Since $H_{k} | H_{k+1}$, $G$ is itself an infinite group of $\partial \mathbb{D}$ as an union of increasing groups, thus $G$ is dense in $\partial \mathbb{D}$ and since $\sigma(R_{HS})$ is a closed subset, we get $\partial \mathbb{D} \subset \sigma(R_{HS})$. 
\end{proof}

The spectrum of $R$ is thus a sequence that goes very fast to 1 whereas the spectrum of $R_{HS}$ is not countable. 

\begin{remark} With the notation above, we can show that if $F$ (the set defined page 4) is more than one line, then $0 \notin \sigma_p(N)$, that is $1 \notin \sigma_p(R)$.
Indeed, if $x=\sum_{k=1}^{\infty} x_k e_k$ is such that $Nx=0$, then 
$$ (\lambda_k-1)x_k+\dfrac{f_k(Px)}{m_{k-1}}=0 \; \; {\rm for} \; \; k \geqslant 3.$$
If $Px \neq 0$, then since $|\lambda_k-1| \leqslant \frac{\pi}{m_k}$, we have 
$$ |x_k| \geqslant \dfrac{|f_k(Px)|m_k}{\pi m_{k-1}} \geqslant \dfrac{15|f_k(Px)|}{\pi}. $$
As $F$ is more than one line, we can find $\epsilon>0$ and a subsequence of $(f_k)$ with $|f_k(Px)| \geqslant \epsilon$. This is a contradiction because $|x_k| \rightarrow 0$ since $(e_k)$ is normalized, so $Px=0$ and $x=0$. If $F$ is just one line (which is a degenerated case for the operator $R$), then it is possible to have $0 \in \sigma_p(N)$. Indeed if $F={\rm span}(e_1)$, then by Proposition \ref{prop:sep}, $F \subseteq {\rm Ker} f_k$ for all $k$, so $f_k(e_1)=0$ and $Ne_1=0$.
\end{remark}

\subsection{Size of the set $\mathcal{A}_X$}

Now, we go back to a more general problem. We have shown that for any separable Banach space, $\mathcal{A}_X$ is non-empty. Hence, a natural question is: what can we say about the size of $\mathcal{A}_X$? It is clear that it cannot be dense for the norm operator because when $R \in \mathcal{A}_X$, $\|R\|>1$. However, we have:

\begin{proposition} \label{prop:dens} If $X$ is a separable Banach space (or more generally satisfies the assumptions of Corollary \ref{cor:main}), then the set $\mathcal{A}_X$ is dense for the strong operator topology on $\mathcal{L}(X)$.
\end{proposition}

To prove this, we shall use the following lemma \cite{BM} (p 45).

\begin{lemma} \label{lem:bm}
Let $X$ be a locally convex topological vector space, and $R \in \mathcal{L}(X)$. Assume that for each integer $N$, one can find $x_1,\ldots,x_N \in X$ such that the vectors $x_1,\ldots,x_N,R(x_1),\ldots,R(x_N)$ are linearly independent. Then ${\rm Sim}(R)$ is SOT dense in $\mathcal{L}(X)$.
\end{lemma}

\begin{proof}[Proof of Proposition \ref{prop:dens}] Recall that the similarity orbit of an operator $R \in \mathcal{L}(X)$ is defined by 
$$ {\rm Sim}(R)=\{JRJ^{-1}, J \in \mathcal{G}\mathcal{L}(X) \}.$$
It is obvious to check that $\mathcal{A}_X$ is invariant under similarity: indeed, with the notations of Theorem \ref{thm:main}, observe that if $S=JRJ^{-1}$, then $x \in A_R \Leftrightarrow Jx \in A_S$ and $x \in B_R \Leftrightarrow Jx \in B_S$. So, it is enough to prove the SOT density of $ {\rm Sim}(R)$ for an operator $R \in \mathcal{A}_X$. For this, we will use our operator $R$ from section 3 (and Lemma \ref{lem:bm}). For example, we check it in the complex case. For $N \geqslant 1$, we put $x_1=e_3+e_4,x_2=e_5+e_6,\ldots,x_N=e_{2N+1}+e_{2N+2}$. Since $Re_k=\lambda_k e_k$ for $k \geqslant 3$, we see that for each $k$, $(x_k,Rx_k)$ is a linearly independant family and that more generally $x_1,\ldots,x_N,R(x_1),\ldots,R(x_N)$ are linearly independant. Hence $ {\rm Sim}(R)$ and therefore $\mathcal{A}_X$ is SOT dense in $\mathcal{L}(X)$. In particular, the set of operators which fail Pr{\v{a}}jitur{\v{a}}'s conjecture is SOT dense in $\mathcal{L}(X)$.
\end{proof}

On the other hand, we have:

\begin{proposition} If $X$ is a complex Banach space, then the set of operators which satisfy Pr{\v{a}}jitur{\v{a}}'s conjecture, that is
$$ \mathcal{P}=\{R \in \mathcal{L}(X), A_R=\emptyset \; \; {\rm or} \; \; A_R \; \; {\rm is} \; \; {\rm dense} \; \; {\rm in} \; \; X\} $$ is dense (for the norm operator) in $\mathcal{L}(X)$. In particular, $\mathcal{A}_X$ has empty interior.
\end{proposition}

\begin{proof} Fix $\epsilon>0$ and $R \in \mathcal{L}(X)$, we can of course assume that $R \notin \mathcal{P}$. The arguments from the begining of this section show that $r(R)=1$. Let $\lambda \in \sigma(R)$ with $|\lambda|=1$. Put $R_{\epsilon}=R+2\epsilon \lambda I$, then $\|R-R_{\epsilon}\|=2\epsilon$ and $\lambda(1+2\epsilon) \in \sigma(R_{\epsilon})$. Hence $r(R_{\epsilon}) \geqslant 1+2\epsilon$ and by the spectral radius formula, we get that for $k$ large enough $\|R_{\epsilon}^k\| \geqslant (1+\epsilon)^k$, so $\sum_{k=1}^{\infty} \frac{1}{\|R_{\epsilon}^k\|}<\infty$ and $R_{\epsilon} \in \mathcal{P}$ because $A_{R_{\epsilon}}$ is dense in $X$ by M\"uller and Vr{\v{s}}ovsk\'y's theorem.
\end{proof}

\section{Acknowledgements} I would like to thank my advisor Robert Deville for suggesting this topic and his help during the elaboration of this paper, Luis Sanchez for fruitful conversations and the anonymous referee for useful comments.

\small{
}

\end{document}